\documentclass[10pt]{amsart}

\usepackage{amsmath,amssymb,amsfonts,epsfig,color}
\input xy
\xyoption{all}

\newcommand{\rb}{\raisebox}
\newcommand{\ig}{\includegraphics}
\newcommand\risS[6]{\rb{#1pt}[#5pt][#6pt]{\begin{picture}(#4,15)(0,0)
  \put(0,0){\ig[width=#4pt]{#2.eps}} #3
     \end{picture}}}

\def\b{\beta}

\def\cM{\mathcal M}
\newtheorem{theorem}{Theorem}[section]
\newtheorem{definition}[theorem]{Definition}
\newtheorem{exa}[theorem]{Example}

\newtheorem{lemma}[theorem]{Lemma}
\newtheorem{prop}[theorem]{Proposition}

\newtheorem{rem}[theorem]{Remark}

\def\wo{\overline}

\def\wt{\widetilde}
\def\R{{\mathbb R}}
\def\Z{{\mathbb Z}}
\def\leq{\leqslant}
\def\geq{\geqslant}
\def\s{\sigma}
\def\bs{\partial\sigma}
\def\ws{\wo\sigma}

\newcommand{\rbetti}[1]{\wt{\b}_{#1}} 
\newcommand{\betti}[1]{\b_{#1}}       
\newcommand{\rH}{{\wt{H}}} 
\newcommand{\homo}[2]{{H}_{#1}(#2)}               

\newcommand{\rank}{\text{rank}}
\newcommand{\tor}{\text{tor}}

\renewcommand{\ker}{\text{Ker}}

\begin{document}

\title[On the Tutte-Krushkal-Renardy polynomial for cell complexes]{On the Tutte-Krushkal-Renardy polynomial for cell complexes}
\author{CARLOS BAJO
, BRADLEY BURDICK, SERGEI CHMUTOV}
\subjclass[2010]{05C10, 05C31, 05C65, 57M15, 57Q15}
\date{}
\address{Department of Mathematics, University of Miami, 
1365 Memorial Drive, Coral Gables, FL 33146. 
{\tt c.bajo@math.miami.edu}}

\address{Department of Mathematics, Ohio State University,
231 West 18th Avenue, Columbus, OH 43210\\
{\tt burdick.28@osu.edu},\qquad
{\tt chmutov@math.ohio-state.edu}}

\keywords{Cell complexes, Tutte polynomial, Krushkal-Renardy polynomial, cellular spanning trees, duality, Bott polynomial.}

\begin{abstract}
Recently V.~Krushkal and D.~Renardy generalized the Tutte polynomial from graphs to cell complexes. We show that evaluating this polynomial at the origin gives the number of cellular spanning trees in the sense of A.~Duval, C.~Klivans, and J.~Martin. Moreover, after a slight modification, the Tutte-Krushkal-Renardy polynomial evaluated at the origin gives a weighted count of cellular spanning trees, and therefore its free term can be calculated by the cellular matrix-tree theorem of Duval et al. In the case of cell decompositions of a sphere, this modified polynomial satisfies the same duality identity as the original polynomial. We find that evaluating the Tutte-Krushkal-Renardy along a certain line gives the Bott polynomial. Finally we prove skein relations for the Tutte-Krushkal-Renardy polynomial.
\end{abstract}

\maketitle

\section*{Introduction} \label{s:intro}
We relate three invariants of cell complexes. The earliest to be introduced was the Bott polynomial first defined by Raoul Bott in 1952 \cite{Bo1,Bo2}. The second invariant we deal with is the number of cellular spanning trees from \cite{DKM1,DKM2}. In these papers, the authors generalize the classical matrix-tree theorem for graphs to arbitrary cell complexes.
G.~Kalai \cite{Ka} first noted that in higher dimensions it makes sense to count spanning trees with weights equal to the square of the order of their codimension one homology groups. Exactly this weighted count of spanning trees is calculated as the determinant of an appropriate submatrix of the Laplacian in \cite{DKM1,DKM2}. Our third invariant is a recent generalization of the Tutte polynomial from \cite{KR}, which we call the {\it Tutte-Krushkal-Renardy polynomial}.

Much of the recent study of the Tutte polynomial from the topological perspective has been influenced by the introduction of the Bollob\'as--Riordan polynomial in \cite{BR3}. In particular, the Tutte-Krushkal-Renardy polynomial from \cite{KR}, has its motivation in a previous work of V.~Krushkal \cite{Kru} about the Tutte polynomial for graphs on surfaces. Indeed, Krushkal and Renardy have also generalized polynomial invariants of graph embeddings in surfaces to arbitrary cell complex embeddings in manifolds \cite{KR}. 

We introduce the Tutte-Krushkal-Renardy polynomial in Section~\ref{s:krushkal} and show that its free term is the number of cellular spanning trees in Section~\ref{s:duval}. In Section~\ref{s:modified}, we modify the Tutte-Krushkal-Renardy polynomial so that its free term becomes the weighted number of cellular spanning trees. We prove an analogous duality identity for the modified polynomial as Krushkal and Renardy did for spheres in section~\ref{ss:duality}. In Section~\ref{s:bott}, we show that the Bott polynomials can be obtained from the Tutte-Krushkal-Renardy polynomial by a substitution. Section~\ref{s:c-d-r} is devoted to the contraction-deletion relations for the Tutte-Krushkal-Renardy polynomial. A few examples and concluding remarks are given in Section~\ref{s:remarks}.

This work has been done as a part of the Summer 2011 undergraduate research working group
\begin{center}\verb#http://www.math.ohio-state.edu/~chmutov/wor-gr-su11/wor-gr.htm#
\end{center}
``Knots and Graphs" 
at the Ohio State University. We are grateful to all participants of the group for valuable discussions, to the OSU Honors Program Research Fund for the student financial support, and to the Summer Research Opportunities Program at OSU for supporting the first named author (Bajo).
The paper was finished during the third author's (Chmutov) visit to the 
Institut des Hautes \'Etudes Scientifiques (IHES) in Bures-sur-Yvette, France.
He thanks IHES for excellent working conditions and warm hospitality and the OSU MRI Stimulus Program for supporting his travel to France.
We are also grateful to Jeremy Martin and Luca Moci for comments to the first version of this paper and to the referee for his or her comments.
\section{The Tutte-Krushkal-Renardy polynomial} \label{s:krushkal}

We refer to \cite{Ha} for the standard notions and facts about CW complexes. We will use the following notation throughout this paper.
\begin {itemize}
\item $K$ is a finite CW complex of dimension $k$; 
\item $K_{(j)}$ denotes the $j$-skeleton of $K$; 
$\mathcal{S}_j$ denotes the set of spanning subcomplexes of dimension $j$, i.e. subcomplexes $S$ such that
$$K_{(j-1)}\subseteq S\subseteq K_{(j)}\ ;$$
\item $f_j(S)$ stands for the number of cells of dimension $j$ in a subcomplex $S$, and we simply write $f_j=f_j(K)$; 
\item $\betti{j}(S)$ denotes the $j^\text{th}$ Betti number, the rank of the homology group $H_j(S;\Z)$;
\item $\rbetti{j}(S)$ denotes the reduced $j^\text{th}$ Betti number, the rank of the reduced homology group $\rH_j(S;\Z)$.
\end{itemize}

\begin{definition}\cite{KR} \rm 
For $1\le j\le k$, the dimension $j$ 
{\it Tutte-Krushkal-Renardy} (TKR)
polynomial, $T_K^j(X,Y)$, is defined by 
$$T^j_K(X,Y)=\sum_{S\in \mathcal{S}_j} X^{\betti{j-1}(S)-\betti{j-1}(K)}Y^{\betti{j}(S)}\ ,
$$ 
Since every spanning complex contains $K_{(j-1)}$, it is useful to identify them with sets of $j$-cells. Thus there are $2^{f_j}$ summands in the sum.
\end{definition}

The dimension $1$ Tutte-Krushkal-Renardy polynomial essentially coincides with
the original Tutte polynomial of the 1-skeleton considered as a graph, $G=K_{(1)}$:
\begin{equation}\label{eq:TP-shift}
T^1_K(X,Y)=T_{G}(X+1,Y+1)\ .
\end{equation}

\bigskip
\begin{definition}\rm  Two cell structures $K$ and $K^*$ on a $k$-manifold $M$ are {\it dual} to each other if there is a one-to-one correspondence between their open cells of complimentary dimensions such that the corresponding $j$-cell $ \s$ of $K$ and $(k-j)$-cell $ \s^*$ of $K^*$ intersect transversely at a single point.
\end{definition}

The cell structure $K^*$ dual to a triangulation $K$ of $M$ can be constructed by setting $\s^*$ to be the union of all simplices of the barycentric subdivision of $K$ intersecting $\s$ only on its barycenter \cite{Ha}. Another way to construct dual cell structure is to use a handle decomposition of $M$. This construction is treated in detail in \cite{RS}.

\begin{theorem}[Duality Theorem for Spheres \cite{KR}] 
Let $K$ and $K^*$ be dual cell structures on $S^k$, then 
$$T^j_{K}(X,Y)=T^{k-j}_{K^*}(Y,X)\ .
$$
\end{theorem}

When $K$ is a (planar) graph embedded in a sphere $S^2$, the theorem becomes the celebrated duality theorem for the Tutte polynomial of a planar graph and its planar dual. 

\begin{rem}\rm
If $K$ is a cell structure on a compact $k$-manifold, then the dimension $k$ Tutte-Krushkal-Renardy polynomial is easily expressed in terms of $f_k$. First for a disjoint union of manifolds, the TKR polynomial is just the product of the TKR polynomials of each component. For every connected, compact $k$-dimensional manifold we may express the TKR polynomial explicitly. 
\begin{equation}\label{eq:TKR-mfld}
T^k_K(X,Y)=
\begin{cases}\displaystyle
Y+\frac{(1+X)^{f_k}-1}{X} &	\text{if }K\text{ is closed orientable,}
              \vspace{10pt}\\
(1+X)^{f_k} & \text{if } K\text{ is non-orientable or has a boundary.}
\end{cases}
\end{equation}

The argument proceeds as follows. 
We partition the $ \mathcal{S}_k$ by the number of $k$-cells in $S\in\mathcal{S}_k$. So there are precisely  $\binom{f_k}{i}$ such $S\in \mathcal{S}_k$ with $i$ $k$-cells. For closed, orientable manifolds, $Y$ appears as the term with $S=K$ whose exponent corresponds to the fundamental class $[K]$, and those subcomplexes $S\neq K$ with $i$ $k$-cells contribute $X^{f_k-i-1}$. 
For other manifolds, those $S\in\mathcal{S}_k$ with $i$ $k$-cells contribute $X^{f_k-i}$. For graphs, the only examples of connected manifolds are the cycle and path graphs, and the above characterizes the Tutte polynomial for those families.
\end{rem}


\section{Cellular Spanning Trees} \label{s:duval}

\begin{definition}\label{def:CST}\rm  \cite{DKM1,DKM2} 
A dimension $j$ {\it Cellular Spanning Tree} (hence $j$-CST)  of $K$ is any $S\in \mathcal{S}_j$ that satisfies the following conditions:\\
(1) $\rH_j(S)=0$,
\qquad\quad
(2) $\rbetti{j-1}(S)=0$,  
\qquad\quad
(3) $ f_j(S)= f_{j}(K) - \rbetti{j}(K_{(j)})+\rbetti{j-1}(K_{(j)})$ \footnote{There is a typo in \cite{DKM1,DKM2} where in the third condition the whole complex $K$ is used instead of $K_{(j)}$.}\ ,\\
 
The second condition implies that the homology group $\rH_{j-1}(S)$ is finite; we will use the notation $|\rH_{j-1}(S)|$ for its order.
\end{definition}

For $j=0$, these conditions mean that $S$ consists of a single point, a vertex
(0-cell) of $K$.
For $j=1$, there is a classic graph theoretical theorem stating that any two of these conditions imply the third one.

\begin{theorem}[``The Two Out of Three Theorem'' \cite{DKM1}]\label{th:2out3}
Let $S\in\mathcal{S}_j$, then any two of the conditions (1), (2), and (3) together imply the third one.
\end{theorem}

For graphs, a spanning tree exists if and only if the graph is connected. We likewise need a condition to consider the existence of a CST. 

\begin{definition}\cite{DKM1}\rm  A CW complex $K$ of dimension $k$ is called \emph{acyclic in positive codimension} (APC) if $\rbetti{j}=0$ for all $j<k$. 
\end{definition}

For example a complex homotopy equivalent to a wedge of several homology spheres of the same dimension is APC. In particular, a connected graph is homotopy equivalent to a wedge of several circles and so is APC.

\begin{theorem}\cite{DKM1} $K$ is APC if and only if $K$ has a $j$-CST for all $j\leq k$.
\end{theorem}

For any fixed $j$, one only needs to assume that $\rbetti{j-1}=0$ for a $j$-CST to exist, but we will assume that $K$ is APC to simplify our hypotheses. 

Henceforth let $\mathcal{T}_j(K)$ denote the set of $j$-CST's of $K$ and let $\tau_j(K)$ denote its cardinality. 
For $K$ being APC we are guaranteed that $\tau_j(K)\neq0$, and we can introduce the invariant inspired by G.~Kalai \cite{Ka}: the number of $j$-CST's, $S$, counted with the weights $|\rH_{j-1}(S)|^2$.

\begin{definition}\cite{DKM2}\rm 
$$\wt{\tau}_j(K) :=\sum_{S\in\mathcal{T}_j(K)} |\rH_{j-1}(S)|^2.$$
\end{definition}

For a connected graph $K$ ($k=1$), the invariant $\wt{\tau}_1(K)$ is equal to the number of its spanning trees because
the group $\rH_0(S)$ is always trivial for connected complexes. The classical matrix tree theorem states that for a graph the number of its spanning trees is equal to a cofactor of the Laplacian associated with a graph. This theorem was generalized to higher dimension in \cite{DKM1,DKM2}. Thus $\wt{\tau}_j(K)$ can be calculated as a determinant of an appropriate matrix.

There is a different generalization of a notion of spanning tree to higher dimension suitable for so called ``Pfaffian Matrix Tree Theorem" \cite{MV1,MV2}.  Their spanning trees are CST's in our sense, but the opposite, in general, is not true.

\subsection{Free term of the Tutte-Krushkal-Renardy polynomial.} \label{ss:KR-free}
A classical evaluation of the Tutte polynomial $T_{K_{(1)}}(1,1)$ gives the number of spanning trees in the graph $K_{(1)}$. Because of the shift of variables in \eqref{eq:TP-shift} it is equal to the evaluation of the first Tutte-Krushkal-Renardy polynomial at the origin,
$T_{K_{(1)}}(1,1)=T^1_K(0,0)$.
Analogously we have the following.

\begin{theorem}\label{th:n-sp}
For an APC complex $K$ and $1\leq j\leq k$,\quad $T^j_K(0,0)=\tau_j(K)$.
\end{theorem}

\begin{proof} 
Note that the exponent of $X$ in the Tutte-Krushkal-Renardy polynomial is equal to $\betti{j-1}(S)-\betti{j-1}(K)=\rbetti{j-1}(S)-\rbetti{j-1}(K)$.
Since $K$ is APC, $\rbetti{j-1}(K)=0$. Then $$T^j_K(X,Y)=\sum_{S\in \mathcal{S}_j} X^{\rbetti{j-1}(S)}Y^{\betti{j}(S)}.
$$
Now the evaluation $T^j_K(0,0)$ is equal to the number of subcomplexes 
$S\in\mathcal{S}_j$ such that $\rbetti{j-1}(S)=0$ and
$\betti{j}(S)=0$. Note that $S$ has dimension less than or equal to $j$, therefore its highest homology group $H_j(S)$ is a free abelian group of rank $\betti{j}(S)=0$. Thus it is trivial, and $\rH_j(S)=0$. Thus conditions (1) and (2) of Definition \ref{def:CST} are satisfied. By theorem \ref{th:2out3}, $S$ must be a cellular spanning tree. In the other direction, $S\in \mathcal{T}_j(K)$ has $\rbetti{j-1}(S)=0$ and $\rH_j(S)=0$. Since $j\geq1$, the reduced homology group $\rH_j(S)$ is isomorphic to the unreduced group $H_j(S)$, and therefore $\betti{j}(S)=0$. Thus $S$ contributes 1 to the evaluation $T^j_K(0,0)$.
\end{proof}

\section{The Modified Tutte-Krushkal-Rendardy Polynomial}\label{s:modified}

For an abelian group $G$, let $\tor(G)$ denote the torsion subgroup of $G$.

\begin{definition}\rm  
We denote the dimension $j$ {\it modified Tutte-Krushkal-Renardy polynomial} of $K$ as $\wt{T}^j_K(X,Y)$, and it is defined as follows.
$$\wt{T}^j_K(X,Y)=\sum_{S\in\mathcal{S}_j} |\tor(\homo{j-1}{S})|^2 X^{\betti{j-1}(S)-\betti{j-1}(K)} Y^{\betti{j}(S)}\ .
$$
\end{definition}

\begin{theorem}\label{th:mn-sp} If $K$ is APC and $j\geq1$, then 
$$\wt{T}^j_K(0,0)=\wt{\tau}_j(K)$$
\end{theorem}

\begin{proof} 
As in the proof of Theorem \ref{th:n-sp} only cellular spanning trees contribute to $\wt{T}^j_K(0,0)$. Only now the contribution of a CST $S$
is equal to $|\tor(\homo{j-1}{S})|^2$, which in turn is just $|\homo{j-1}{S}|^2$.
\end{proof}

\begin{rem}\rm
The modified Tutte-Krushkal-Renardy polynomial is an example of a  quasi-arithmetic Tutte polynomial from \cite{BM,DAM} with the multiplicity function $m(S):=|\tor(\homo{j-1}{S})|^2$ associated with matroid $\cM$ from Remark \ref{rem-mat}. One can show that this multiplicity function satisfies the axioms (A1) and (A2) of quasi-arithmetic matroids from \cite{BM}. If the
kernel of the boundary map $\partial_j: C_j(K;\R)\to C_{j-1}(K;\R)$ is at most 2-dimensional (that is the nullity of the matroid $\cM$ is at most 2), then this multiplicity function satisfies also the axiom (P) and the quasi-arithmetic matroid $\cM$ is in fact arithmetic. However, we do not know
whether the axiom (P) is satisfied in general.
\end{rem}

\bigskip
\subsection{Duality Theorem for the modified Tutte-Krushkal-Renardy polynomial}\label{ss:duality}

\begin{theorem}[The Duality Theorem for $\wt{T}^j_K(X,Y)$]\label{th-mdual}
If $K$ and $K^*$ are dual cell decompositions of $S^k$ and if $1\leq j\leq k-1$, then $$\wt{T}^j_K(X,Y)=\wt{T}^{k-j}_{K^*}(Y,X).
$$
\end{theorem}

In the proof of this theorem we use a technical lemma from Krushkal and Renardy.

{\bf Lemma} (\cite{KR}).
{\it If $K$ is a cell decomposition of $S^k$ and $S\subseteq K$ is a subcomplex of $K$, then $S$ is homotopy equivalent to $S^k \setminus S^*$, where $S^*$ is a subcomplex of the dual cell decomposition $K^*$ formed by cells which do not intersect $S$.}

\begin{proof} We invoke the universal coefficient theorem, Alexander duality (see \cite{Ha}), and the above Lemma 
to get the following isomorphisms.
$$
\rH_j(S)/\tor(\rH_j(S))\oplus\tor(\rH_{j-1}(S)) \cong \rH^j(S)
 \cong \rH_{k-j-1}(S^{k}\setminus S)
 \cong \rH_{k-j-1}(S^*).
$$
Which gives us the following identities.
$$\betti{j}(S)\!=\!\rbetti{j}(S)\!=\!\rbetti{k-j-1}(S^*),\quad
\rbetti{j-1}(S)\!=\!\rbetti{k-j}(S^*)\!=\!\betti{k-j}(S^*), \quad
\tor(\rH_{j-1}(S))\!\cong\!\tor(\rH_{k-j-1}(S^*)).
$$

These together will conclude the proof. 
\begin{align*}
\wt{T}_K^j(X,Y) & =\sum_{S\in \mathcal{S}_j} |\tor(\homo{j-1}{S})|^2 X^{\rbetti{j-1}(S)-\rbetti{j-1}(K)}Y^{\betti{j}(S)}
   \\
& =\sum_{S^*\in \mathcal{S}_{k-j}^*} |\tor(\homo{k-j-1}{S^*})|^2 
    X^{\betti{k-j}(S^*)}Y^{\rbetti{k-j-1}(S^*)-\rbetti{k-j-1}(K)}\\
& =\wt{T}_{K^*}^{k-j}(Y,X)
\end{align*}
We used the equations $\rbetti{j-1}(K)=0$ and $\rbetti{k-j-1}(K)=0$ for $K=S^k$ and $1\leq j\leq k-1$.
\end{proof}

\begin{rem}\rm
Theorem \ref{th-mdual} claims that self-dual cell decompositions of $S^k$ have symmetric Tutte polynomials. Examples of self-dual cell decompositions of $S^k$ can be obtained from the $k$-skeletons of $(k+1)$-dimensional regular self-dual polytopes. In 3 dimensions there is only one example, the tetrahedron. In 4 dimensions, there is the 4-simplex. There is also one that has Schl\"afli symbol $\{3,4,3\}$ with 24 facets associated with the root system of the simple Lie algebra $F_4$, see \cite{Cox}. Additionally, Peter McMullen, in his PhD thesis, found a different self-dual cell decomposition of $S^3$ with 120 facets, see \cite{BCM}.
It would be interesting to find further relations of Theorem \ref{th-mdual} with the paper \cite{Max}.
\end{rem}

\section{The Bott polynomial} \label{s:bott}

In 1952 Raoul Bott introduced two polynomials for CW complexes  \cite{Bo1,Bo2} invariant under subdivisions, called  
{\it combinatorially invariant}. Due to a mistake in computing the second polynomial for the sphere, Bott erroneously claimed that they are independent. In fact, Wang shows in \cite[Proof of Theorem 4.2]{Wa} that they are proportional to each other after a suitable change of variables as well as exhibiting an entire class of invariant polynomials in \cite{Wa}. Thus we essentially have only one Bott polynomial which for a finite $k$ dimensional cell complex $K$ can be defined as
$$R_K(\lambda) := \sum_{S\in\mathcal{S}_k} 
(-1)^{f_k(K)-f_k(S)} \lambda^{\b_k(S)}\ .
$$

If $K$ is an orientable manifold without boundary, then $R_K(\lambda)=\lambda-1$, see \cite{Bo1}. 
Z.~Wang \cite{Wa} observed that for graphs, $k=1$,  the coefficients of the Bott polynomial essentially coincide with the {\it Whitney numbers} \cite{Wh} and, in the case of planar graphs, the Bott polynomial is equal to the  chromatic polynomial of the dual graph. Thus the Bott polynomial of a graph is equal to its {\it flow polynomial} (see its definition in \cite{Bo}). Therefore one may regard the Bott polynomial as a higher dimensional generalization of the flow polynomial of graphs. Different approaches to a higher dimensional flow polynomial were suggested in \cite{BeKe,DKM3,Go}. It would be interesting to find a relation of these approaches to the Bott polynomial.

\begin{theorem}\label{bott}
$$R_K(\lambda) = (-1)^{\b_k(K)} T_K^k(-1,-\lambda)\ .$$
\end{theorem}

\begin{proof} 
The Euler characteristics of the $(k-1)$-skeleton $K_{(k-1)}$
(which is contained in both $K$ and $S$) in terms of the numbers of cells gives the equation.
$$\chi(K_{(k-1)}) = 
\chi(K) - (-1)^k f_k(K) = \chi(S) - (-1)^k f_k(S)\ .
$$
The same computation in terms of the Betti numbers gives the following.
$$\chi(K_{(k-1)}) = 
\chi(K) - (-1)^{k-1}\b_{k-1}(K) - (-1)^k\b_k(K) = 
\chi(S) - (-1)^{k-1}\b_{k-1}(S) - (-1)^k\b_k(S)\ .
$$
Subtracting these two equations we get
$$(-1)^{k-1}\b_{k-1}(K) + (-1)^k \b_k(K) - (-1)^k f_k(K) = (-1)^{k-1}\b_{k-1}(S)+ (-1)^k\b_k(S) - (-1)^k f_k(S)\ .
$$
Dividing by $(-1)^{k-1}$ gives the following.
\begin{equation}\label{eq:betti-k-1}
\b_{k-1}(S)-\b_{k-1}(K)  = f_k(K)-f_k(S) - \b_k(K) + \b_k(S)\ .
\end{equation}

Now the monomial on the right hand side of the theorem's identity corresponding to a subcomplex $S$ is
$$(-1)^{\b_k(K)} (-1)^{\b_{k-1}(S) - \b_{k-1}(K)}(-\lambda)^{\b_k(S)} = 
(-1)^{f_k(K)-f_k(S)}\lambda^{\b_k(S)}\ .
$$
This coincides with the corresponding monomial of the left hand side.
\end{proof}

\section{Skein relations for the Tutte-Krushkal-Renardy polynomial}\label{s:c-d-r}

For graphs, the Tutte polynomial is often equivalently defined by a set of contraction-deletion relations, which we call here {\it skein relations} following the knot theoretic terminology. Z.~Wang \cite{Wa} found skein (contraction/deletion) relations for the Bott polynomial. Here we generalize them to the Tutte-Krushkal-Renardy polynomial.

For a $k$-cell $\s$ of $K$, we denote its closure in $K$ and its boundary in $K$ by $\ws$ and $\bs$ respectively. As subcomplexes of $K$, $\ws$ and $\bs$ inherit the CW structures from $K$. The following definitions generalize the standard definitions for graphs. It was motivated by matroid theory, see Remark \ref{rem-mat}.

\begin{definition}\label{def:l-b} \rm 
\begin{itemize}
\item $\s$ is a \emph{loop} in $K$ if $H_k(\ws) \cong  \Z$\ ; 

\item $\s$ is a \emph{bridge} in $K$ if 
$\b_{k-1}(K\setminus\s) = \b_{k-1}(K)+1$\ ;

\item $\s$ is \emph{boundary regular} if  $\rH_{k-1}(\bs)\cong\Z$\ .
\end{itemize}
\end{definition}

%
%

\begin{prop}\label{cor:lb-prop}
{\bf (a)} A loop is not a bridge.

{\bf (b)} If $\s$ is a bridge for $K$, then it is a bridge for any spanning subcomplex $S\ni\s$.
\end{prop}

\begin{proof}
Let $\s$ be a loop. Then $C_k(\ws)=\Z$ is generated by $\s$, so if $H_k(\ws)=\Z$, $\partial_k(\s)=0$. So $\partial_k(C_k(K\setminus\s))= \partial_k(C_k(K))$.
Thus, $\b_{k-1}(K\setminus\s)= \rank(\ker(\partial_{k-1}))-\rank(\partial_k(C_k(K\setminus\s))) = \rank(\ker(\partial_{k-1}))-\rank(\partial_k(C_k(K)))=\b_{k-1}(K)$ and $\s$ is not a bridge. 

If $\s$ is a bridge for $K$, then 
$\rank(\partial_k(C_k(K\setminus\s)))= \rank(\partial_k(C_k(K)))- 1$.
This means that $\partial_k(\s)$ is independent from the images of all other $k$-cells of $K$. In particular, it is independent from the images of the $k$-cells of $S$ different from $\s$. Consequently, 
$\rank(\partial_k(C_k(S\setminus\s)))= \rank(\partial_k(C_k(S)))- 1$, which means that $\s$ is a bridge for $S$.
\end{proof}

\begin{exa}\label{ex:contact-loop}
\rm  In graphs, every loop has a single vertex as a boundary and so they are not boundary regular. When we consider cell complexes other than graphs, loops can suddenly be

\noindent
\parbox[t]{3.8in}{ boundary regular. Let $K$ be a 2-sphere with two points identified to a single point $p$. It has a CW structure consisting of one 2-cell $\s$, one 1-cell (edge) $e$, and one 0-cell $p$. The closure $\ws$ coincides with the whole complex $K$ which has a homotopy type of the wedge $S^2\vee S^1$. All its homology groups are isomorphic, 
$H_2(K)\cong H_1(K)\cong H_0(K)\cong\Z$. Thus $\s$ is a loop. 
On the\linebreak}\qquad\qquad
\risS{-70}{s2vs1}{\put(18,52){$\s$}\put(34,30){$e$}\put(81,35){$p$}
                 }{80}{0}{0} \vspace{-8pt}\\
other  hand,
$\bs=e\cup p= S^1$. So $\rH_1(\bs)\cong\Z$, and $\s$ is boundary regular.
 \end{exa} 

For the following theorem we use standard tools from algebraic topology such as the long exact sequence of a pair and the fact that a CW complex $X$ and its subcomplex $A$ form a ``good pair", so $H_i(X,A)\cong\rH_i(X/A)$. 
We refer to \cite{Ha} for all of these facts. In particular, for a $k$-cell $\s$ we can consider the quotient space $K/\ws$ as a CW complex obtained from $K$ by collapsing all cells in $\ws$ to a point.

\begin{theorem}\label{th:skein}
The Tutte-Krushkal-Renardy polynomial satisfies the following relations:

{\bf (i)} If $\s$ is neither a bridge nor a loop and is boundary regular, then 
$$T^k_K(X,Y)= T^k_{K/\ws}(X,Y) + T^k_{K\setminus\s}(X,Y)\ .$$

{\bf (ii)} If $\s$ is a loop, then 
$$T^k_K(X,Y) = (Y+1) T^k_{K\setminus\s}(X,Y)\ .$$

{\bf (iii)} If $\s$ is a bridge and boundary regular, then 
$$T^k_K(X,Y) = (X+1) T^k_{K/\ws}(X,Y)\ .$$
\end{theorem}

\bigskip
We will use the following lemma about contraction.

\begin{lemma} For $S\ni\sigma$, if $\sigma$ is boundary regular and not a loop, then $\b_k(S)=\b_k(S/\ws)$ and $\b_{k-1}(S)-\b_{k-1}(K)=\b_{k-1}(S/\ws)-\b_{k-1}(K/\ws)$. \end{lemma}
\begin{proof}

First consider the case $k=1$. Then $K$ is a graph, and the claim is that the contraction of an edge, $e$, that is not a loop, i.e. has distinct vertices, preserves the number of circuits of $K$ and the number of connected components for every subgraph. Indeed, any circuit containing $e$ has other edges, since it is not a loop, so contracting $e$ will only reduce the size of the circuit. And contraction clearly cannot separate any subgraph. 

Now, we assume that $k>1$. Since $\s$ is not a loop, $H_k(\ws)=0$. Since $\s$ is boundary regular,  $H_{k-1}(\bs)=\rH_{k-1}(\bs)\cong\Z$. These two conditions are equivalent to the condition $(C)$ of \cite[Theorem 4.1]{Wa}. 

Consider the long exact sequence of a pair $(\ws,\bs)$
$$H_k(\bs) \longrightarrow H_k(\ws)\longrightarrow 
  \rH_k(\ws/\bs)\longrightarrow H_{k-1}(\bs) \longrightarrow 
  H_{k-1}(\ws) \longrightarrow \rH_{k-1}(\ws/\bs)\ .
$$
We have $H_k(\bs)=0$, $H_k(\ws)=0$,  $H_{k-1}(\bs)\cong\Z$, $\rH_k(\ws/\bs)\cong\Z$, and $\rH_{k-1}(\ws/\bs)=0$ because $\ws/\bs$ is a $k$-sphere. So the sequence becomes
$$0\longrightarrow 0\longrightarrow \Z\longrightarrow \Z\longrightarrow 
  H_{k-1}(\ws) \longrightarrow 0\ .
$$
The exactness of this sequence implies that $H_{k-1}(\ws)$ is a finite group. Thus $\b_{k-1}(\ws)=0$.

Now for a subcomplex $S$, consider the long exact sequence of a pair $(S,\ws)$
$$H_k(\ws) \longrightarrow H_k(S)\longrightarrow 
  H_k(S/\ws)\longrightarrow H_{k-1}(\ws) 
$$
Tensoring it by the field of real numbers $\R$ we get
$$0 \longrightarrow H_k(S;\R)\longrightarrow 
  H_k(S/\ws;\R)\longrightarrow 0\ . 
$$
Which means that $\b_k(S)=\b_k(S/\ws)$ for all $S\ni\s$, and in particular for
$S=K$. Then, using the equation \eqref{eq:betti-k-1} from the proof of theorem \ref{bott}, we have
\begin{align*}
\b_{k-1}(S)-\b_{k-1}(K) &= f_k(K)-f_k(S) - \b_k(K) + \b_k(S) \\
&= f_k(K/\ws)-f_k(S/\ws) - \b_k(K/\ws) + \b_k(S/\ws)
= \b_{k-1}(S/\ws)-\b_{k-1}(K/\ws)\ .
\end{align*}
\end{proof}
\bigskip
\begin{proof}[Proof of Theorem \ref{th:skein}] 
To prove the theorem we partition the set of all top dimensional subcomplexes $S$ according to the property $S\ni\s$ or $S\not\ni\s$.

\bigskip
In case (i), the sum over all $S\not\ni\s$ gives $T^k_{K\setminus\s}(X,Y)$ because 
$\b_{k-1}(K\setminus\s)=\b_{k-1}(K)$ since $\s$ is not a bridge. The lemma above implies that the sum over all $S\ni\s$ is equal to $T^k_{K/\ws}(X,Y)$.  This proves part (i) of the theorem.

\bigskip
In case (ii), the sum over all $S\not\ni\s$ is again $T^k_{K\setminus\s}(X,Y)$ for the same reason,
i.e. a loop $\s$ is not a bridge according to
Proposition \ref{cor:lb-prop}. 

For $S\ni\s$ we have $\partial_k(\s)=0$. 
So $\partial_k(C_k(S\setminus\s))= \partial_k(C_k(S))$. Therefore the chain complex for $S$ is isomorphic to the direct sum of the chain complex for 
$S\setminus\s$ and the chain complex 
$0 \longrightarrow \Z\longrightarrow 0$ with $\Z$ at the grading $k$.
Thus we get
$\b_{k-1}(S)=\b_{k-1}(S\setminus\s)$ and $\b_k(S)=\b_k(S\setminus\s)+1$.
Consequently the sum over $S\ni\s$ is equal to $YT^k_{K\setminus\s}(X,Y)$ which proves part (ii).

\bigskip
For case (iii), the lemma gives us 
that the sum over all $S\ni\s$ is equal to $T^k_{K/\ws}(X,Y)$.
The subcomplexes $S\not\ni\s$ are in 1-to-1 correspondence with the subcomplexes $(S\cup\s)\ni\s$. We prove that under this correspondence the sum over all $S\not\ni\s$ is equal to $X$ times the sum all $(S\cup\s)\ni\s$ which is $T^k_{K/\ws}(X,Y)$.

According to 
Proposition \ref{cor:lb-prop}, if $\s$ is a bridge for $K$ it is a bridge for any subcomplex containing $\s$, in particular for $S\cup\s$. Then 
$\b_{k-1}(S)=\b_{k-1}(S\cup\s)+1$ for any subcomplex $S\not\ni\s$. This gives an extra $X$ in the sum over all such $S$ compared to the sum over $S\cup\s$. To compare the exponents of $Y$ consider the cell chain complexes of $S$ and $S\cup\s$:
$$\xymatrixcolsep{0pc}\xymatrixrowsep{-2pt}\xymatrix{
0\longrightarrow &C_k(S)\ar[drrrrr]^-{\partial_k}\ar@{^{(}->}[dd] \\
&&&&&&C_{k-1}(K)\longrightarrow \dots \\
0\longrightarrow &C_k(S\cup\s)\hspace{-18pt}\ar@<1ex>[urrrrr]^-{\partial_k}}
$$
Since $\s$ is a bridge of $S$, its image $\partial_k(\s)$ is independent from
the images of all other $k$-cells of $S\cup\s$. Therefore any element of $\ker(\partial_k|_{C_k(S\cup\s)})$ actually belongs to $C_k(S)$. This means that 
$H_k(S\cup\s)=\ker(\partial_k|_{C_k(S\cup\s)}) = 
\ker(\partial_k|_{C_k(S)}) = H_k(S)$.
Consequently, $\b_k(S\cup\s)=\b_k(S)$ which proves case (iii) of the theorem.
\end{proof}

\section{Examples and concluding remarks.} \label{s:remarks}

\begin{exa}\rm  
The CW complex $K$ is obtained by a modification of Example \ref{ex:contact-loop} gluing an 

\noindent
\parbox[t]{3.8in}{additional 2-cell $\s'$ along the edge $e$.
It is homotopy equivalent to a sphere $S^2$. So its homology groups are
$H_1(K)=0$ and $H_2(K)\cong H_0(K)\cong\Z$. 
The cell $\s$ is still a boundary regular loop.
The new cell $\s'$ is a bridge since $\b_1(K\setminus\s) = 1$. And it is also boundary regular. The skein relation (iii) of Theorem \ref{th:skein}
gives\\
\hspace*{2cm} $T^2_K(X,Y) = (X+1) T^2_{K/\ws'}(X,Y)\ .$
}\qquad\qquad
\risS{-70}{s2vs2}{\put(18,52){$\s$}\put(25,25){$e$}\put(81,35){$p$}
                  \put(48,45){$\s'$}
                 }{80}{0}{64}\\
 
The complex $K/\ws'$ consists of one 2-cell $\s$ and one 0-cell $p$. It is homeomorphic to a sphere $S^2$. By definition its Tutte-Krushkal-Renardy polynomial is equal to $T^2_{K/\ws'}(X,Y)=Y+1$, which is also in agreement with the formula \eqref{eq:TKR-mfld}. Thus we have
$$T^2_K(X,Y) = (X+1)(Y+1) = XY+X+Y+1\ .$$

Alternatively, by the skein relation (ii), we may delete the loop $\sigma$ from $K$ giving 
$$ T^2_K(X,Y)=(Y+1)T^2_{K\setminus\sigma}(X,Y).$$
The complex $K\setminus\sigma$ consists of one 2-cell $\s'$, one 1-cell $e$, and one 0-cell $p$. It is homeomorphic to a disc $D^2$. By definition, $T^2_{K\setminus \s}(X,Y)=X+1$. Thus again we have that
$$T^2_K(X,Y)=(Y+1)(X+1)=YX+Y+X+1\ .$$
\end{exa} 

\begin{exa}\label{ex:s2-3}\rm  
Just as every loop in a graph is not boundary regular, every non-loop in a graph \emph{is} boundary regular, i.e. they have two distinct vertices. The following example shows that this is no longer the case for cell complexes other than graphs. Consider the following cell structure $K$ on a 2-sphere represented as a plane together with a point at infinity. It has three 2-cells $\s_1,$

\noindent
\parbox[t]{3.8in}{ $\s_2,\s_\infty$, three 1-cells $a,b,c$, and two vertices (0-cells) $p,q$. Note that $\s_\infty$ is not boundary regular since its boundary $\bs$ coincides with 1-skeleton of $K$ and has a homotopy type of a wedge of two circles. Thus $H_1(\bs)\cong\Z^2$. It is also neither a bridge nor a loop. The next table shows the contribution of the various subcomplexes $S$ into the Tutte-Krushkal-Renardy polynomial $T^2_K(X,Y)$. 
}\qquad\quad
\risS{-65}{s2-3}{\put(45,55){$\s_\infty$}\put(16,30){$\s_1$}
  \put(72,30){$\s_2$}\put(37,25){$p$}\put(56,25){$q$}
  \put(20,50){$a$}\put(48,35){$b$}\put(80,49){$c$}
                 }{100}{0}{75}

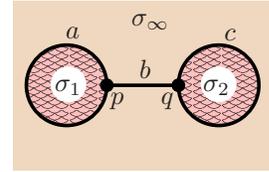
$$\begin{array}{|c||c|c|c|c| c|c|c|c|}\hline
S&\varnothing&\{\s_1\}&\{\s_2\}&\{\s_1,\s_2\} \rb{-5pt}{\makebox(0,15){}} 
    &\{\s_\infty\}&\{\s_1,\s_\infty\}&\{\s_2,\s_\infty\}
    &\{\s_1,\s_2,\s_\infty\} \\ \hline
&X^2&X&X&1 \rb{-5pt}{\makebox(0,15){}} &X&1&1&Y\\ \hline
\end{array}
$$

\medskip\noindent
Therefore, $T^2_K(X,Y)=X^2+3X+3+Y$ which agrees with the equation \eqref{eq:TKR-mfld} since $K=S^2$ is a manifold. The deletion
$K\setminus\s_\infty$ consists of two discs connected by a segment $b$.
Its Tutte-Krushkal-Renardy polynomial is equal to $T^2_{K\setminus\s_\infty}(X,Y)=X^2+2X+1$. The contraction $K/\ws_\infty$ is a wedge of two spheres, so 
$T^2_{K/\ws_\infty}(X,Y)=1+2Y+Y^2$. So the boundary regular condition is essential for the skein relation. On the other hand, $\s_1$ and $\s_2$ satisfy the conditions of case (i) of Theorem \ref{th:skein}, and one may check that 
$T^2_K(X,Y)= T^2_{K/\ws_i}(X,Y) + T^2_{K\setminus\s_i}(X,Y)$ for $i=1,2$.
\end{exa} 

\begin{exa}\rm  
Let $K=\R P^2$ with the standard CW structure: one 2-cell $\s$, one 1-cell $e$, and one 0-cell $p$. In this case $\s$ is a bridge since its deletion gives the circle $\R P^1$. It is also boundary regular because $\bs$ is the same circle. We have $T^2_K(X,Y)=X+1$, while $T^2_{K/\ws}(X,Y)=1$ since $K/\ws$ is a point. And so $T^2_K(X,Y)=(X+1)T^2_{K/\ws}(X,Y)=(X+1).$ Note that this coincides with equation (\ref{eq:TKR-mfld}) with $k=2$, since $f_2=1$. 
\end{exa} 
 
\begin{rem}\rm
In \cite[Theorem 4.1]{Wa} the contraction/deletion relation for the Bott polynomial $R_K(\lambda)$ was obtained. However the skein relations (ii) and (iii) of Theorem \ref{th:skein} were not there.
We would like to mention here that analogues of (ii) and (iii) for $R_K(\lambda)$ follows from Theorem \ref{bott}.
\end{rem}

\begin{rem}\rm
The action of contracting a bridge is not the only operation that one might consider for the skein relation (iii). The importance of contraction is that it preserves the top homology group of a space, but the caveat is that it might greatly affect other homology groups. Indeed, this is why we must require a bridge to be {\it boundary regular}, and by that we really mean that it does not affect the codimension one homology group too radically. In many instances a bridge will be a cell with a ``free face", i.e. an independent vector of the boundary. In such an instance, there is an operation called {\it collapsing} in which the cell and its free face are deleted. Collapsing, unlike contraction, is (when permissible) a homotopy equivalence and therefore affects no homology group. This is stronger than what is required for part (iii) of the theorem to hold, so whenever a bridge has a free face one may collapse rather than contract. Upon collapsing one will have a space homotopic to the original, and therefore it is more desirable than contraction. For cases where the bridge is not boundary regular but still has a free face collapsing will still work for the theorem. 

For example, consider the CW complex $K^\prime=K\setminus\{\sigma_1,\sigma_2\}$ for $K$ from Example  \ref{ex:s2-3}. The 2-cell 

\noindent
\parbox[t]{3.8in}{$\sigma_\infty$ is a bridge that is not boundary regular. As mentioned, the skein relation (iii) fails for $\sigma_\infty$ under contraction, but if we apply (iii) with collapsing (say with respect to face $c$), the resulting complex  is just $\overline{a\cup b}$. We have that $T^2_{K^\prime}(X,Y)=(X+1)T^2_{\overline{a\cup b}}(X,Y)$. Now since $\overline{a\cup b}$ has no 2-cells, there is one summand in $T^2_{\overline{a\cup b}}(X,Y)$, which is necessarily $1$. Thus $T^2_{K^\prime}(X,Y)=X+1$, which coincides with equation 
\eqref{eq:TKR-mfld} for manifolds with boundary.
}\qquad\qquad
\risS{-65}{s2-3a}{\put(45,55){$\s_\infty$}\put(37,25){$p$}
  \put(56,25){$q$}\put(20,50){$a$}\put(48,35){$b$}\put(80,49){$c$}
                 }{100}{0}{75}\\ 
\end{rem}

\begin{rem}\label{rem-topdim}\rm
Theorem \ref{th:skein} gives the skein relation for the top dimensional, $k$-th, Tutte-Krushkal-Renardy polynomial. However, a lower dimensional polynomial 
$T^j_K(X,Y)$ is proportional to $T^j_{K_{(j)}}(X,Y)$, so essentially it depends only on the $j$-th skeleton of $K$. Thus we have a skein relation for them as well. However one has to be careful with the deletion of a $j$-cell $\s$ for $j<k$: the resulting topological space $K\setminus\s$ might not be a cell complex anymore.
\end{rem}

\begin{rem}\label{rem-mat}\rm
As indicated in \cite{KR}, the Tutte-Krushkal-Renardy polynomial is the Tutte polynomial of a matroid $\cM$ obtained in the following way.  Consider the $j$-th chain group of $K$ with real coefficients, $C_j(K;\R)$.
As a vector space over $\R$ it has a distinguished basis formed by the $j$-cells $\s_i$. Consider the matrix of $\partial_j:C_j(K;\R)\longrightarrow C_{j-1}(K;\R)$ with columns indexed by this basis of $C_j(K;\R)$. If $\cM$ is the column matroid of $\partial_j$, then Krushkal and Renardy \cite{KR} showed that $T^j_K(X,Y)=T_\cM(X+1,Y+1)$. 
On matroid theory, we refer to two excellent books \cite{Ox,Wel} and a pioneering paper \cite{Wh2}. 

Our definitions \ref{def:l-b} of a loop and of a bridge are designed in such a way that the corresponding element of the matroid $\cM$ will be a loop or a bridge respectively. Moreover, if a cell $\s$ is boundary regular, then the  matroid of the chain complex of $K/\ws$ is the matroid obtained from $\cM$ by the matroid theoretic contraction of the corresponding element. It was Wang's proof \cite[Lemma 4.1]{Wa} that demonstrated the necessity of the \emph{boundary regularity} condition. The equivalence of contraction and deletion in $K$ with contraction and deletion $\cM$ gives an alternative proof of \ref{th:skein}. 

In Example \ref{ex:s2-3} the ground set of the matroid $\cM$ consist of three vectors $\partial_2(\s_1)=a$, $\partial_2(\s_2)=c$, and  $\partial_2(\s_\infty)=-a-c$ in 3-space $\R^3=\langle a,b,c \rangle$, and one relations between them: $\partial_2(\s_1)+\partial_2(\s_2)+\partial_2(\s_\infty)=0$. The matroid theoretical contraction of an element $\partial_2(\s_\infty)\in\cM$ would give a matroid on two elements which are dependent. So the rank $\cM/\partial_2(\s_\infty)$ is equal to 1. Meanwhile, the topological contraction of the cell $\s_\infty$ would give a wedge of two spheres. The corresponding matroid would consist of two zero vectors, since the boundary of each of the two cells consist of the single point. Its rank would be 0.

As noted, an APC complex is a generalization of a connected graph, and $\cM$ is a generalization of the graphic matroid of a connected graph. Since the spanning trees of a connected a graph are exactly the bases of the graphic matroid, one might suspect that the CST's of an APC complex are exactly the bases of $\cM$. Indeed, Peterson \cite{Pe} shows that this is the case. And if one wishes to discard the APC requirement just as one may wish to consider disconnected graphs, Peterson shows that if we weaken condition (2) of Definition \ref{def:CST} to $\rbetti{j-1}(S)=\rbetti{j-1}(K)$, that the resulting ``$j$-trees" \cite[Propisition 5.1]{Pe} are exactly the bases of $\cM$ for any complex. One may think of the weakening of this condition for arbitrary complexes as the weakening of the notion of spanning trees to maximal spanning forests for nonconnected graphs. 

Since we know that $T_K^j(X-1,Y-1)=T_\cM(X,Y)$ and that CST's essentially correspond to bases of $\cM$, we can define internal and external activities with respect to a CST by means of the matroid $\cM$. And we then have that
 $$T^j_K(X-1,Y-1)=\sum_{S\in \mathcal{T}_j}X^{\iota(S)}Y^{\epsilon(S)}.$$
Where $\epsilon(S)$ and $\iota(S)$ correspond respectively to the number of external activities and internal activities with respect to $S$. This then implies that the coefficients of $T_K^j(X-1,Y-1)$ are nonnegative, meaning the TKR polynomial is a sort of 
generating function. As we have demonstrated, it counts CST's. 
It should also satisfy a type of combinatorial reciprocity.
\end{rem}

\begin{rem}\label{rem-4var}\rm
When $K$ is embedded in a closed, oriented $(2k)$-manifold, the intersection pairing on $H_k(K)$
gives another numerical invariant of $S\in \mathcal{S}_k$ that depends on the embedding $S\hookrightarrow K$. In \cite{KR} this invariant is used to extend the TKR polynomial to 4 variables in a way analogous to the way the Tutte polynomial for graphs on surfaces is extended in \cite{Kru}. It would be interesting to phrase a contraction-deletion relation for this polynomial.

\end{rem}


\bigskip
\begin{thebibliography}{ABCD}

\bibitem[BK]{BeKe} M.~Beck, Y.~Kemper, {\it Flows on Simplicial Complexes}. 
   24th International Conference on Formal Power Series and Algebraic
   Combinatorics (FPSAC 2012), Discrete Mathematics \& Theoretical Computer     Science (2012) 817-–826.

\bibitem[Bo]{Bo} B.~Bollob\'as, {\it Modern graph theory},
   Graduate Texts in Mathematics~{\bf 184}, Springer, New York, 1998.

\bibitem[BR]{BR3} B.~Bollob\'as and O.~Riordan, {\it A polynomial of graphs
   on surfaces}, Math. Ann. {\bf 323} (2002) 81--96.

\bibitem[BCM]{BCM} J.~Bokowski, P.~Cara, and S.~Mock, {\it On a self 
   dual 3-sphere of Peter McMullen}, 
   Periodica Mathematica Hungarica {\bf 39} (1999) 17-32. 

\bibitem[Bo1]{Bo1} R.~Bott, {\it Two new combinatorial invariants for
   polyhedra},
   Portugualiae Mathematica, {\bf 11} (1952) 35--40.

\bibitem[Bo2]{Bo2} R.~Bott, {\it Reflections of the theme of the posters},
   in {\it Topological Methods in Modern Mathematics}, A Symposium in Honor 
   of John Milnor's Sixtieth Birthday, Goldberg and Phillips (1993) 125--135.

\bibitem[BM]{BM} P.~Br\"and\'en, L.~Moci, {\it The multivariate arithmetic Tutte polynomial}. 
   Preprint \verb#arXiv:1207.3629 [math.CO]#.  To appear on Transactions Am. Math. Soc.

\bibitem[Cox]{Cox} H.~S.~M.~Coxeter, {\it Regular polytopes},
   Dover (1973).

\bibitem[DAM]{DAM} M.~D'Adderio, L.~Moci, {\it Arithmetic matroids, 
   Tutte  polynomial, and toric arrangements}, 
   Advances in Mathematics, {\bf 232}(1) (2013) 335--367

\bibitem[DKM1]{DKM1} A.~Duval, C.~Klivans, J.~Martin, 
   {\it Cellular matrix-tree theorems},
   Trans. Amer. Math. Soc., {\bf 361}(11) (2009) 6073--6114.

\bibitem[DKM2]{DKM2} A.~Duval, C.~Klivans, J.~Martin, 
   {\it Cellular spanning trees and Laplacians of cubical complexes},
   Advances in Applied Mathematics, {\bf 46} (2011) 247--274.

\bibitem[DKM3]{DKM3} A.~Duval, C.~Klivans, J.~Martin, 
   {\it Cuts and flows of cell complexes},
   Preprint \verb#arXiv:1206.6157 [math.CO]#.
   
\bibitem[Go]{Go} L.~Godkin, {\it Aspheric Orientations of Simplicial Complexes}, MA Thesis San Francisco State University (2012).

\bibitem[Ha]{Ha} A.~Hatcher, {\it Algebraic Topology}, 
   Cambridge University Press (2009).

\bibitem[Ka]{Ka} G.~Kalai, {\it Enumeration of Q-acyclic cellular 
   complexes}, Israel J. Math. {\bf 45} (1983) 337--351.

\bibitem[Kr]{Kru} V.~Krushkal, {\it Graphs, Links, and Duality on Surfaces}, 
   Combinatorics, Probability and Computing {\bf 20} (2011) 267--287. 

\bibitem[KR]{KR} V.~Krushkal, D.~Renardy, {\it A polynomial invariant and 
   duality for triangulations}. 
   Preprint \verb#arXiv:1012.1310v2 [math.CO]#.

\bibitem[MV1]{MV1} G.~Masbaum, A.~Vaintrob, {\it A new matrix-tree  
   theorem}, 
   Int. Math. Res. Not. {\bf 27} (2002) 1397�-1426.

\bibitem[MV2]{MV2} G.~Masbaum, A.~Vaintrob, {\it Milnor numbers, Spanning 
   Trees, and the Alexander-Conway Polynomial}, 
   Adv. Math. {\bf 180} (2003) 765�-797.
   
\bibitem[Max]{Max} M.~Maxwell,  {\it Enumerating bases of self-dual 
   matroids},
   J. Combin. Theory, Ser. A, {\bf 116} (2009), 351--378.

\bibitem[Ox]{Ox} J.~Oxley, {\it Matroid Theory}, Oxford University Press 
   (2011).
   
\bibitem[Pe]{Pe} A.~Peterson, {\it Enumerations of spanning trees in  simplicial complexes}, U.U.D.M. Report {\bf 13} (2009), 1-57.

\bibitem[RS]{RS} C.~Rourke, B.~Sanderson, {\it Introduction to 
   Piecewise-linear Topology}, Springer-Verlag (1972).
   
\bibitem[Wa]{Wa} Z.~Wang, {\it On Bott polynomials}. 
   Journal of Knot Theory and Its Ramifications, {\bf 3}(4) (1994) 537--546. 

\bibitem[Wel]{Wel} D.~J.~A.~Welsh, {\it Matroid Theory}, 
   Academic Press (1976) and Dover (2010).

\bibitem[Wh1]{Wh} H.~Whitney, {\it A set of topological invariants for 
   graphs}. Amer.~J.~Math. {\bf 55} (1933) 231--235. 

\bibitem[Wh2]{Wh2} H.~Whitney, {\it On the abstract properties of linear dpendence}, Amer.~J.~Math. {\bf 57}(3) (1935) 509--533.

\end{thebibliography}
\end{document}